\documentclass[11pt]{article}
\usepackage[english]{babel}
\usepackage{amsmath}
\usepackage{amsthm}
\usepackage{amssymb}
\usepackage[mathscr]{euscript}
\usepackage{color}

\newtheorem{theorem}{Theorem}[section]
\newtheorem*{thmA}{Theorem A}
\newtheorem*{thmB}{Theorem B}

\newtheorem{definition}[theorem]{Definition}
\newtheorem{lemma}[theorem]{Lemma}

\newtheorem{proposition}[theorem]{Proposition}
\newtheorem{corollary}[theorem]{Corollary}
\newtheorem{remark}[theorem]{Remark}

\newcommand{\uxi}{\underline{\xi}}

\newcommand{\pp}{\partial}
\newcommand{\rr}{\mathbb{R}}
\newcommand{\unx}{\underline{x}}
\newcommand{\ux}{\underline{x}}

\newcommand{\uy}{\underline{y}}

\newcommand{\uo}{\underline{\omega}}

\newcommand{\al}{{\alpha}}

\newcommand{\mC}{{\mathbb{C}}}
\newcommand{\mR}{{\mathbb{R}}}

\newcommand{\mZ}{{\mathbb{Z}}}
\newcommand{\mN}{{\mathbb{N}}}

\newcommand{\cC}{{\mathcal{C}}}
\newcommand{\cM}{{\mathcal{M}}}

\newcommand{\cH}{{\mathcal{H}}}
\newcommand{\cS}{{\mathcal{S}}}

\newcommand{\cK}{{\mathcal{K}}}
\newcommand{\cD}{{\mathcal{D}}}

\newcommand{\la}{{\lambda}}

\title{\bf  The Radon transform between monogenic and generalized slice monogenic functions}

\author{F. Colombo\\Dipartimento di Matematica\\Politecnico di
Milano\\Via Bonardi, 9, 20133 Milano, Italy,\\ fabrizio.colombo@polimi.it \and  R. L\'avi\v cka \\ Faculty of Mathematics and Physics\\ Charles University in Prague\\
Sokolovsk\'a 83, Praha, Czech Republic\\ lavicka@karlin.mff.cuni.cz\\\and I. Sabadini\\Dipartimento di Matematica \\Politecnico di
Milano\\Via Bonardi, 9, 20133 Milano, Italy, \\ irene.sabadini@polimi.it \and V. Sou\v{c}ek\\ Faculty of Mathematics and Physics\\ Charles University in Prague\\
Sokolovsk\'a 83, Praha, Czech Republic\\ soucek@karlin.mff.cuni.cz}

\date{\today}
\begin{document}
\maketitle
\begin{abstract}
In \cite{bls}, the authors describe a link between holomorphic  functions depending on a parameter and monogenic functions defined on  $\mR^{n+1}$ using the Radon and dual Radon transforms.
The main aim of this paper is to further develop this approach. In fact, the Radon transform for functions with values in the Clifford algebra $\mR_n$  is mapping  solutions of the generalized Cauchy-Riemann equation, i.e., monogenic functions, to a~parametric family of holomorphic functions with values in $\mR_n$ and, analogously, the dual Radon transform is mapping parametric families of holomorphic functions as above to monogenic functions.
The parametric families of holomorphic functions considered in the paper can be viewed as a~generalization of the so-called slice monogenic functions. An important part of the problem solved in the paper is to find a suitable definition of the function spaces serving as the domain and the target of both integral transforms.
 \end{abstract}

\section{Introduction}

The study of solutions of the Dirac (resp. Weyl) equations is very well-established research field, usually called Clifford analysis (see \cite{bds,dss,ghs}). Following ideas of Cullen (\cite{C}),  a theory of slice monogenic functions was developed recently (\cite{slicecss,css}) as an alternative.
Both these theories are natural generalizations of the complex function theory but they are quite different from each other. Possible relations between the two theories
were  developed in odd dimensions in the context of the Fueter construction (which allows to construct monogenic functions starting from holomorphic functions) and its inversion (\cite{CoSaSo1,CoSaSo2}). The possibility to relate the two theories in even dimensions, using the Fueter mapping technique is still under investigation.\\
In \cite{bls}, there are some ideas to find a~link between slice monogenic functions and monogenic functions defined on $\mR^{n+1}$ using the Radon and dual Radon transforms.
The main aim of this paper is to further develop this approach.

In fact, the Radon transform for $\mR_n$-valued functions (where $\mR_n$ denotes the universal Clifford algebra over $n$ imaginary units $e_1,\ldots,e_n$) is mapping solutions of the generalized Cauchy-Riemann equation, i.e., monogenic functions, to a~parametric family of holomorphic functions with values in $\mR_n$ and, analogously, the dual Radon transform is mapping parametric families of holomorphic functions as above to monogenic functions.
The most important feature of these transforms,  for our aims, is the fact that they map slice monogenic functions to monogenic ones (and viceversa). So we can study their action just on initial data, i.e., restrictions to the hyperplane $x_0=0$ in $\mR^{n+1}$.

On the level of polynomial solutions, the classical Fischer decomposition for polynomials  allows us to consider just the initial data of the form $\ux^m P_k(\ux)$ for some homogeneous monogenic polynomial $P_k$ of degree $k$. The main fact is that the dual Radon transform acts on the initial data of the above type as a~multiple of the identity.

Using the inversion, the correspondence can be extended to  monogenic functions, which are polynomials in a~neighborhood of infinity. We are led to consider their Laurent expansion and the initial data of the form $|\ux|^{-(n+m)} P_k(\ux)$ or $|\ux|^{-(n+m)}\ux P_k(\ux)$ for some homogeneous monogenic polynomial $P_k$ of degree $k$.
Again, the Radon transform acts on the initial data of the above type as a~multiple of the identity.

The discussion of the Radon and the dual Radon transform on the level of initial data for polynomial
solutions is based on facts discussed
already in the papers by F. Sommen (\cite{sommen4,sommen2,sommen3}).
It is quite important, however, to extend both integral transforms to appropriate complete function spaces.
On the side of monogenic functions, the best choice
is the set (in fact a right $\mathbb R_n$-module) of 'entire' monogenic functions, i.e., monogenic functions defined on the whole space
$\mR^{n+1}.$ The attempt to define the corresponding space on the other side leads to a generalization of the standard definition of slice monogenic functions introduced in the Section 3.
The use of inversion leads to the function spaces
defined at infinity and the relation between the Taylor
and Laurent series.
The main results of the paper are Theorems A and B describing
the action of both integral transforms on the corresponding function spaces.\\
It is important to note that this approach allows to describe relations between slice monogenic functions and monogenic functions which are independent of the parity of the dimension $n$.

The plan of the paper is the following. In Section 2, we formulate the main results of the paper, namely Theorems A and B. In Section 3, we collect properties of monogenic and slice monogenic functions and, in particular,  we deal with their Laurent series expansions.  In Section 4, we recall the dual Radon transform and give a~proof of Theorem A. Finally, in Section 5, we recall the Radon transform and prove Theorem B.

\vskip 2mm
\noindent
{\bf The acknowledgement.} The three co-authors (FC, IS and VS) thank the E. \v Cech Institute (the grant P201/12/G028 of the Grant Agency of the Czech Republic)
for the support during the preparation of the paper.

\section{Main results}

In order to state  the main results of this paper, we have to introduce some function spaces.
The setting in which we will work is the real Clifford algebra
$\rr_n$ over $n$ imaginary units $e_1,\dots,e_n$ satisfying the
relations
 $$
 e_ie_j+e_je_i=-2\delta_{ij}.
 $$
We identify an element $(x_1,x_2,\ldots,x_n)\in\rr^n$ with a so called 1-vector in the Clifford algebra
through the map $$(x_1,x_2,\ldots,x_n)\mapsto \unx=
x_1e_1+\ldots+x_ne_n.
$$
By  $S^{n-1}$ we denote the sphere of unit 1-vectors in $\mathbb{R}^n$, i.e.
$$
S^{n-1}=\{ \uo=e_1\omega_1+\ldots +e_n\omega_n\ |\  \omega_1^2+\ldots +\omega_n^2=1\}.
$$

\begin{definition}[see \cite{bds}]
Let $U\subseteq \mathbb{R}^{n+1}$ be an open set. Then
a real differentiable function $f:U\to \mathbb{R}_n$ is called (left) monogenic if $\mathcal{D}f =0$ where
$\mathcal{D}$ is the generalized Cauchy-Riemann operator in $\mR^{n+1}$
$$
\mathcal{D}=\partial_{x_0}+\sum_{j=1}^n e_j\partial_{x_j}.
$$
\end{definition}

We denote by $\cM$  the right $\mathbb R_n$-module of monogenic functions on $\mR^{n+1}\setminus\{0\}$ and by $\cM_0$
the submodule of functions in $\cM$ which extend analytically to $\mR^{n+1}$.
It is well known that monogenic functions are real analytic and
\[
\cM=\cM_0\oplus \cM_\infty
\]
where $\cM_\infty$ is the right-submodule of functions $f\in\cM$ which extend analytically at $\infty$. Function theory of monogenic functions is called Clifford analysis. For an account of Clifford analysis, we refer to \cite{bds, bls, dss, GM, ghs}.

It is useful to note that the vector subspace $\mathbb{R}+\underline{\omega}\mathbb{R}$ of $\mR^{n+1}$ passing through $1$ and
$\underline{\omega}\in S^{n-1}$, denoted by $\mathbb{C}_{\underline{\omega}}$,  is a complex plane whose elements
are $x_0+\underline{\omega}p$, for $x_0$, $p\in
\mathbb{R}$. Now we introduce slice monogenic functions, namely, functions holomorphic
on each complex plane $\mC_{\underline{\omega}}$.

\begin{definition}
Let $U\subseteq\mR^2$ be an open set and let $f:\ U\times S^{n-1}\to\mR_n$ be a~smooth function. Then $f$ is called a~(generalized) slice monogenic function on $U\times S^{n-1}$
if it satisfies the equation
$\partial_{\uo}\; f (x_0,p,\underline{\omega})=0$ for any $\underline{\omega}\in S^{n-1}$ where
$$
\partial_{\uo}=\frac 12(\partial_{x_0}+\underline{\omega}\partial_p)
$$
is the Cauchy-Riemann operator on the plane $\mC_{\uo}$.
\end{definition}
\begin{remark}{\rm In this context, a~function $f(x_0,p,\uo)$ is even if $f(x_0,-p,-\uo)=f(x_0,p,\uo)$.
In what follows, we are restricting ourselves to even functions because the planes $\mC_{\underline{\omega}}$ and
$\mC_{-\underline{\omega}}$ are obviously the same.}
\end{remark}

\noindent
We denote by $\cS$ the right $\mathbb R_n$-module of even slice monogenic functions on $(\mR^2\setminus\{0\})\times S^{n-1}$ and by $\cS_0$ the submodule of functions in $\cS$ which extends analytically to $\mR^2\times S^{n-1}$.

In Theorem \ref{t-Laurent}, it is shown that $\cS=\cS_0\oplus \cS_\infty$
where $\cS_\infty$ is the submodule of functions $f\in\cS$ which extend analytically at $\infty$.
Moreover, $\cS_\infty$ is the image of $\cS_0$ under the inversion $I_{2}$ defined by
$$
 [I_2(f)](x_0,p,\uo)=\frac{x_0-\uo p}{x_0^2+p^2}
 f\left(\frac{x_0}{x_0^2+p^2},\frac{-p}{x_0^2+p^2},\uo\right).
$$
This is the usual inversion for functions defined on the complex plane extended to parametric families of holomorphic functions that we consider in this paper.

\begin{remark}{\rm
Note that our definition of slice monogenic functions is a~generalization of the usual one in the literature, but it is more suitable for our purposes.
In the previous works in the literature, see \cite{slicecss, css}, it turns out that the values $f (x_0,0,\underline{\omega})$ on the common real axis $p=0$ does not depend on $\uo$.
As it appears clear from the definition, slice monogenic functions in our sense can be viewed as holomorphic maps with values in the Clifford algebra $\mR_n$ depending on the parameter
$\uo\in S^{n-1}$.
}
\end{remark}

The first main result shows how the dual Radon transform relates the set $\cS_0$ of entire slice monogenic functions and the set $\cM_0$
of entire monogenic functions.
The dual Radon transform $S$ of a~smooth function $f$ on $\mR^2\times S^{n-1}$ is defined by
$$
S[f](x_0,\ux):=\frac{1}{A_n}\int_{S^{n-1}}f(x_0,(\ux,\uo),\uo)\;d\uo,
$$
where $(\cdot,\cdot)$ denotes the scalar product in $\mR^n$, $d\uo$ is the area element on the sphere $S^{n-1}$ and $A_n$ is the area of $S^{n-1}$.
Note that the variable $x_0$ plays a~role of parameter with respect to the transform $S$.
 As we will see in Lemma 4.2,  $S[\partial_{\uo}f]=\cD(S[f])$ and hence the transform $S$ maps slice monogenic functions into monogenic ones.
In our first main result, we summarize properties of the transform $S$ when acting on $\cS_0$, in particular, we describe its range and the~complement of its kernel ${\rm ker}(S)$.

We now introduce a~submodule $\cS\cM_0$ of the right $\mathbb R_n$-module $\cS_0$ which turns out to be the~complement of the kernel ${\rm ker}(S)$.
It is well-known that each entire monogenic function $f\in\cM_0$ is uniquely determined by its restriction $f|_{x_0=0}$.
Here $f|_{x_0=0}(\ux)=f(0,\ux),\ \ux\in\mR^n.$ Analogously,  slice monogenic functions in $\cS_0$ are uniquely determined by their restriction to $x_0=0$.

\begin{definition}\label{d-SM0}
Let $\underline{\cM}_0$ stand for the set of restrictions $f|_{x_0=0}$ of functions $f\in\cM_0$.
Let us denote by $\cS\cM_0$ the set of functions $f\in\cS_0$ such that
$$f(0,p,\uo)=f_0(p\uo),\ (p,\uo)\in\mR\times S^{n-1}$$
for some function $f_0\in\underline{\cM}_0$.
\end{definition}
By definition, a~function $f\in\cS_0$ belongs to $\cS\cM_0$ if and only if its restriction to $x_0=0$ is a~function of $\underline{\cM}_0$
when we use spherical coordinates $\ux=p\uo$.\\
 Now we are ready to state the first main result of the paper.

\begin{thmA}
The dual Radon transform $S$ is a~linear map of the $\mathbb R_n$-module $\cS_0$ onto $\cM_0$.
Moreover, $\cS_0$ can be decomposed as
$$\cS_0 ={\rm ker}(S)\oplus \cS\cM_0.$$
In particular, $S$ is an isomorphism between $\cS\cM_0$ and $\cM_0$.
\end{thmA}

\begin{remark}
{\rm
For a~proof of Theorem A, see Section \ref{s-dRt}.
For a~description of the kernel ${\rm ker}(S)$  and of the set $\cS\cM_0$ in terms of Taylor series expansions, we refer to Corollary \ref{c-kerS} and Corollary \ref{c-SM0}, respectively.
}
\end{remark}

In order to state our second main result,
let us recall that the Radon transform is defined as follows.
Let $f$ be an $\mR_n$-valued function defined in $\mR^{n+1}.$
For $(x_0,p,\uo)\in\mR^2\times S^{n-1}$, we define
$$
R[f](x_0,p,\uo):=\int_{L(\uo,p)} f(x_0,\ux) \;d\sigma(\ux)
$$
whenever the integral exists.
Here $L(\uo,p)=\{\ux\in\mR^n|\ (\ux,\uo)=p\}$ and $d\sigma $ is the Lebesgue measure on the hyperplane $L(\uo,p)$.
Note again that the variable $x_0$ plays a~role of parameter and $R$ is the Radon transform with respect to the variable $\ux\in\mR^n$.
  Under natural assumptions, see \cite{GGV}, $\partial_{\uo}(R[f])=R[\cD f]$ and hence the transform $R$ maps monogenic functions into slice monogenic ones.
Indeed, we prove the following result.

\begin{thmB}
The Radon transform $R$ is an injective linear map of the $\mathbb R_n$-module $\cM_\infty$ into  $\cS_\infty$.
Denote $\cS\cM_{\infty}=I_2(\cS\cM_0)$. Then $R(\cM_\infty)=\cS\cM_{\infty}$
and $R$ is an isomorphism between $\cM_\infty$ and $\cS\cM_{\infty}$.
Moreover, we have that
$$\cS_\infty=I_2({\rm ker}(S))\oplus \cS\cM_{\infty}.$$
\end{thmB}

\begin{remark}
{\rm
For a~proof of Theorem B, see Section \ref{s-Rt}.
For a~description of the $\mathbb R_n$-modules $I_2({\rm ker}S)$ and $\cS\cM_{\infty}$ in terms of Laurent series expansions, we refer to Corollary \ref{c-kerS} and Corollary \ref{c-SM8}, respectively.
}
\end{remark}

\section{Function spaces}

In this section, we summarize some properties of monogenic and slice monogenic functions that we need later on.

\subsection{Monogenic functions}

First we recall the Cauchy-Kovalevskaya extension for monogenic functions.
 As we already discussed, each monogenic function $f:\mR^{n+1}\to\mR_n$ (that is, $f\in\cM_0$) is uniquely determined by its restriction $f|_{x_0=0}$.
If $f_0\in\underline{\cM}_0$, the set of restrictions $f|_{x_0=0}$ of functions $f\in\cM_0$, then we write $f=CK(f_0)$ for the~unique function $f\in \cM_0$ such that $f|_{x_0=0}=f_0$.
Then the Cauchy-Kovalevskaya  extension operator (CK-extension for short) $CK$ is an isomorphism of the $\mathbb R_n$-module $\underline{\cM}_0$ onto $\cM_0$.

In what follows, we need to know explicitly the CK-extension of polynomials
of the form $\underline{x}^j P_k(\underline{x})$ where  $P_k:\mR^n\to\mR_n$ is a~spherical monogenic of degree $k$.
To this end, let us recall that the Gegenbauer polynomial $C^{\nu}_j$ is defined as
\begin{equation}\label{gegenbauer}
C^{\nu}_j(z)=\sum_{i=0}^{[j/2]}\frac{(-1)^i(\nu)_{j-i}}{i!(j-2i)!}(2z)^{j-2i}\text{\ \ with\ \ }
(\nu)_{j}=\nu (\nu+1)\cdots (\nu+j-1).
\end{equation}
Actually, we need the following result which is proved in \cite[p. 312, Theorem 2.2.1]{dss}.

\begin{lemma}\label{lckxj}
Let $j\in\mN_0$ and $P_k$ be a spherical monogenic of degree $k$. Then we have that
$$CK(\ux^j P_k(\ux))=X^{(j)}_k(x) P_k(\underline{x})$$
where $X^{(0)}_k=1$ and, for $j\in\mN,$ the polynomial $X^{(j)}_k$ is given by
\begin{equation}\label{La31}
X^{(j)}_k(x)=
\mu^j_k|x|^j\left (C_j^{(n-1)/2+k}\left(\frac{x_0}{|x|}\right)+\frac{n+2k-1}{n+2k+j-1}C_{j-1}^{(n+1)/2+k}\left(\frac{x_0}{|x|}\right)\frac{\underline{x}}{|x|}\right )
\end{equation}
with $\mu^{2l}_k=(-1)^l(C_{2l}^{(n-1)/2+k}(0))^{-1}$ and $$\mu^{2l+1}_k=(-1)^l\frac{n+2k+2l}{n+2k-1}(C_{2l}^{(n+1)/2+k}(0))^{-1}.$$
\end{lemma}

To give an estimate of the embedding factors $X^{(j)}_k$ we use the following simple lemma.

\begin{lemma}\label{l-estimate}
Let $m\in\frac12\mN_0$, $a_1\in(0,1)$ and $a_2\in(2,\infty)$. Then there are $C_1,C_2\in(0,\infty)$ such that
\begin{equation}\label{e-est_Gamma}
C_1a_1^{j+k}\leq
\frac{\Gamma(j+k+1)}{\Gamma(j+1)\Gamma(k+1+m)}
\leq C_2a_2^{j+k}
\end{equation}
for all $j,k\in\frac12\mN_0.$
\end{lemma}

\begin{proof}
Because of the estimate
$$
\frac{\sqrt{\pi}}{2}\leq\frac{\Gamma(j+1/2)}{\Gamma(j)}\leq j\sqrt{\pi},\
j\in\mN,
$$
we can limit ourselves to $j,k,m\in\mN_0$. Then \eqref{e-est_Gamma} follows from
$$\frac{\Gamma(j+k+1)}{\Gamma(j+1)\Gamma(k+1+m)}=\frac{1}{(k+m)\cdots(k+1)}\binom{j+k}{j}$$ and
$1\leq\binom{j+k}{j}\leq\binom{j+k}{0}+\binom{j+k}{1}+\cdots+\binom{j+k}{j+k}=2^{j+k}$.
\end{proof}

\begin{lemma}\label{l_est_EF} Let $X_k^{(j)}$ be as in \eqref{La31}.
There is a~constant $b>0$ such that $$|X^{(j)}_k(x)|\leq b^{k+j}\; |x|^j,\ x\in\mR^{n+1},$$ for all $k,j\in\mN_0$.
\end{lemma}

\begin{proof}
It follows easily from Lemma \ref{l-estimate}
and the following estimates
$$|C_{2l}^{\nu}(0)|=\frac{\Gamma(\nu+l)}{\Gamma(l+1)\Gamma(\nu)}
\mbox{\ \ and\ \ }
|C_j^{\nu}(t)|\leq C_j^{\nu}(1)=\frac{\Gamma(2\nu+j)}{\Gamma(j+1)\Gamma(2\nu)}\mbox{\ \ for\ \ } |t|\leq 1.$$
\end{proof}

Now we describe Taylor series expansions for entire monogenic functions.

\begin{theorem}\label{t-Taylor_m}
A~function $f:\mR^{n+1}\to\mR_n$ belongs to $\cM_0$ if and only if it has a~unique expansion of the form
\begin{equation}\label{e-Taylor_m}
f(x)=\sum_{j,k\in\mN_0} X^{(j)}_k(x) P_{k,j}(\ux)
\end{equation}
for some $k$-homogeneous monogenic polynomials $P_{k,j}:\mR^n\to\mR_n$ satisfying
\begin{equation}\label{Pkj}
\limsup_{j+k\to\infty} \sqrt[j+k]{\|P_{k,j}\|}=0.
\end{equation}
Here $\|\cdot\|$ is the $L_2$-norm $\|\cdot\|_{L_2}$ or the supremum norm $\|\cdot\|_{\infty}$ on the sphere $S^{n-1}$, that is,
$$\|g\|_{L_2}^2=\frac{1}{A_n}\int_{S^{n-1}}|g(\uo)|^2\;d\uo\mbox{\ \ and\ \ }\|g\|_{\infty}=\sup\{|g(\uo)|;\ \uo\in  S^{n-1} \}.$$
\end{theorem}

\begin{proof}
It is well-known that the condition \eqref{Pkj} for the norm $\|\cdot\|_{L_2}$ is equivalent to that for the norm $\|\cdot\|_{\infty}$, see e.g.\ \cite[p. 28]{mor2}. Indeed, $P_{k,j}$ are (componentwise) $k$-homogeneous spherical harmonics.
Moreover,  \eqref{Pkj} is obviously equivalent to the condition that,
for any  $a>0$, there exists a~positive constant $C$ such that
\begin{equation}\label{Pkj_2}
\|P_{k,j}\|\leq C\, a^{k+j}
\end{equation}
for all $j,k\in\mN_0$.

\medskip\noindent
(i) Let $f$ have an expansion \eqref{e-Taylor_m}. Then the series in \eqref{e-Taylor_m} converges absolutely and locally uniformly on $\mR^{n+1}$. Indeed, by Lemma \ref{l_est_EF} and \eqref{Pkj_2}, we have that
$$|X^{(j)}_k(x) P_{k,j}(\ux)|\leq b^{k+j}|x|^j|\ux|^k \|P_{k,j}\|_{\infty}\leq (b|x|)^{k+j}\|P_{k,j}\|_{\infty}.$$
All the summands in \eqref{e-Taylor_m} are monogenic on $\mR^{n+1}$ and so is the sum $f$.

\medskip \noindent
(ii) By \cite{sommen4}, a~given function $f\in\cM_0$ has a~unique expansion of the form
\begin{equation}\label{e-Homog_m}
f(x)=\sum_{k=0}^{\infty} f_k(x)
\end{equation}
for some $k$-homogeneous monogenic polynomials $f_k:\mR^{n+1}\to\mR_n$ satisfying
\begin{equation}\label{fk}
\limsup_{k\to\infty} \sqrt[k]{\|f_k\|^*}=0.
\end{equation}
Here $\|\cdot\|^*$ is the $L_2$-norm $\|\cdot\|^*_{L_2}$ or the supremum norm $\|\cdot\|^*_{\infty}$ on the sphere $S^{n}$.
By the monogenic Fischer decomposition (see \cite[Sect. 1.10]{dss}), we have that
$$f_k(\ux)=\sum_{j=0}^k \ux^j P_{k-j,j}(\ux)$$
for some $(k-j)$-homogeneous monogenic polynomials $P_{k-j,j}:\mR^n\to\mR_n$ and, by applying the CK extension operator, we get
\begin{equation}\label{branching}
f_k(x)=CK(f_k(\ux))=\sum_{j=0}^k X^{(j)}_{k-j}(x) P_{k-j,j}(\ux),
\end{equation}
which gives the expansion \eqref{e-Taylor_m}. It only remains to verify the estimate \eqref{Pkj}. By \eqref{fk}, for a~given $a>0$,
there is $C>0$ such that
$$Ca^k\geq \|f_k\|^*_{L_2}\geq \|X^{(j)}_{k-j}(x) P_{k-j,j}(\ux)\|^*_{L_2},\ k\geq j\geq 0.$$
Here we use the fact that the decomposition \eqref{branching} is orthogonal with respect to the $L_2$-inner product, see \cite[Theorem 2.2.3, p.\ 315]{dss}.
As we know we have an analogous estimate (with possibly different constant $C$) for the supremum norm. Hence, we can assume that, for all $k\geq j\geq 0$,
$$Ca^k\geq \|X^{(j)}_{k-j}(x) P_{k-j,j}(\ux)\|^*_{\infty}\geq \|\ux^j P_{k-j,j}(\ux)\|_{\infty}=\|P_{k-j,j}\|_{\infty},$$
which concludes the proof.
\end{proof}

To obtain Laurent series expansions for monogenic functions of $\cM_{\infty}$ we use the inversion.
It is well-known that the inversion $I_{n+1}$, defined by
$$
[I_{n+1}(f)](x)=\frac{\bar{x}}{|x|^{n+1}}f\left(\frac{x}{|x|^2}\right),
$$
is an isomorphism of the $\mathbb R_n$-module $\cM_0$ onto $\cM_{\infty}$.
Hence a~function $f$ belongs to ${\cM}_\infty$ if and only if it has an expansion of the form
\begin{equation}\label{e-Laurent_m}
f(x)=\frac{\bar x}{|x|^{n+1}}\sum_{j,k\in \mN_0} \frac{X^{(j)}_k(x)  P_{k,j}(\underline{x})}{|x|^{2(j+k)}}
\end{equation}
where $P_{k,j}$ are as in \eqref{Pkj}.
Obviously, each function $f\in\cM_\infty$ is again uniquely determined by its initial datum $f|_{x_0=0}$.

\subsection{Slice monogenic functions}

First we describe Laurent series expansions for slice monogenic functions.

\begin{theorem}\label{t-Laurent}

\begin{itemize}
\item[(i)]
A~function $f$ belongs to $\cS$ if and only if it has a~unique expansion of the form
\begin{equation}\label{LaurentS}
f(x_0,p,\uo)=\sum_{j\in\mZ,k\in \mN_0}(x_0+\uo p)^j \uo^k P_{k,j}(\uo)
\end{equation}
for some $k$-homogeneous monogenic polynomials $P_{k,j}:\mR^n\to\mR_n$ which satisfy the following condition:
For each $m\in\mN_0$ and each $a>0$, there is $C>0$ such that
\begin{equation}\label{Pkj_weak}
\|P_{k,j}\|\leq C\, a^{|j|} (k+1)^{-m}
\end{equation}
for all $k\in\mN_0$ and $j\in \mZ$.
Here $\|\cdot\|$ is the $L_2$-norm  $\|\cdot\|_{L_2}$ or the supremum norm $\|\cdot\|_{\infty}$ on the sphere $S^{n-1}$.

\item[(ii)]
Moreover, a~function $f$ belongs to $\cS_0$ (resp.\ $\cS_{\infty}$) if and only if it has a~unique expansion of the form \eqref{LaurentS}
with $P_{k,j}=0$ unless $j\geq 0$ (resp.\ $j<0$). In particular, we have that $\cS=\cS_0\oplus\cS_{\infty}$ and $\cS_{\infty}=I_2(\cS_0)$.
\end{itemize}

\end{theorem}

\begin{proof}
(i) If a~function $f$ has an expansion of the form \eqref{LaurentS} then  $f$ obviously belongs to $\cS$. Indeed, by \cite[Theorem 4]{see},
the series in  \eqref{LaurentS} converges in the $\cC^{\infty}$ sense.

So let us assume that $f\in\cS$. Let $\uo\in S^{n-1}$ be fixed.
Then it is well-known that the Clifford algebra $\mR_n$ is a~finite dimensional vector space over $\mC_{\uo}$.
Thus the values of $f$ can be decomposed into  $\mC_{\uo}$-valued components which, by our assumptions,  are holomorphic functions on $\mC_{\uo}\setminus\{0\}$ in the usual sense (cfr. with \cite[Splitting Lemma]{css})
and they have unique Laurent series expansions. Hence we can expand $f$ uniquely as
\begin{equation}\label{LaurentS_1}
f(x_0,p,\uo)=\sum_{j\in\mZ}(x_0+\uo p)^j C_j(\uo)
\end{equation}
where the coefficients $C_j(\uo)$ are given by the Cauchy integral
\begin{equation}\label{Cauchy}
C_j(\uo)=\frac{1}{2\pi r^j}\int_0^{2\pi} (\cos(jt)-\uo\sin(jt))\; f(r\cos t,r\sin t,\uo) \; dt
\end{equation}
for any $r>0$.
In particular, the coefficients $C_j(\uo)$ are smooth functions on $S^{n-1}$.
Denote by $\Delta_{\uo}$ the Laplace-Beltrami operator on the sphere $S^{n-1}$.
Let $j\in\mZ$, $m\in\mN_0$ and $r>0$ be given. Then, by applying $\Delta_{\uo}^m$ to \eqref{Cauchy}, we obtain the estimates
\begin{equation}\label{CauchyEst}
\|\Delta_{\uo}^m C_j\|_{L_2}\leq\|\Delta_{\uo}^m C_j\|_{\infty}\leq D r^{-j}
\end{equation}
with $D=D(m)=\sup\{|\Delta_{\uo}^m f(x_0,p,\uo)|+|\Delta_{\uo}^m\; \uo f(x_0,p,\uo)|;\ x_0^2+p^2=r^2,\ \uo\in S^{n-1}  \}<\infty$.
Since $C_j$ is a~smooth function on $S^{n-1}$ it is well-known that it has a~unique expansion
\begin{equation}\label{HarmExp}
C_j(\uo)=\sum_{k=0}^{\infty} Y_{k,j}(\uo)
\end{equation}
into $k$-homogeneous spherical harmonics $Y_{k,j}$ which decrease rapidly as $k\to\infty$, and satisfying that
\begin{equation}\label{Ykj}
k^{2m}\|Y_{k,j}\|_{L_2}\leq\|\Delta_{\uo}^m C_j\|_{L_2},\ k,m\in\mN_0,
\end{equation}
see e.g.\ \cite[p. 36]{mor2}.
By \eqref{CauchyEst} and \eqref{Ykj}, it follows that the sequence $Y_{k,j}$ satisfy the estimate \eqref{Pkj_weak}.
Since $f$ is an even function so are the coefficients $C_j$ and, consequently, only $Y_{k,j}$ for even $k$ are non-zero.
Let $j\in\mZ$ and let $k\in\mN_0$ be even. Then it is well-known (see \cite[Cor. 1.33]{dss}) that $k$-homogeneous spherical harmonic $Y_{k,j}:S^{n-1}\to\mR_n$
decomposes uniquely as
\begin{equation}\label{YkjPkj}
Y_{k,j}(\uo)=P_{k,j}(\uo)+\uo P_{k-1,j}(\uo)=\uo^k (-1)^{k/2}P_{k,j}(\uo)+\uo^{k-1} (-1)^{(k-2)/2} P_{k-1,j}(\uo)
\end{equation}
for some spherical monogenics $P_{k,j}$ and $P_{k-1,j}$ of degree $k$ and $k-1$, respectively. Since this decomposition is orthogonal with respect to the $L_2$-inner product
we have that
$$\|Y_{k,j}\|_{L_2}^2=\|P_{k,j}\|_{L_2}^2+\|P_{k-1,j}\|_{L_2}^2,$$
which completes the proof of the statement (i). Indeed, for the given function $f\in\cS$, we found an expansion of the form \eqref{LaurentS}, see \eqref{LaurentS_1}, \eqref{HarmExp} and \eqref{YkjPkj}.

\medskip\noindent
(ii) It follows directly from (i). In particular, we have that $I_2((x_0+\uo p)^j)=(x_0+\uo p)^{-(j+1)}$ for $j\in\mZ$.
\end{proof}

\begin{corollary}\label{c-SM0}
The following statements are equivalent.
\begin{itemize}
\item[(i)]
A~function $f$ belongs to $\cS\cM_0$.
\item[(ii)]
A~function $f$ has a~unique expansion of the form \eqref{LaurentS} with $P_{k,j}=0$ unless $j\geq k$.
\item[(iii)]
A~function $f$ has a~unique expansion of the form
\begin{equation}\label{Taylor_sm}
f(x_0,p,\uo)=\sum_{j,k\in \mN_0}(x_0+\uo p)^{j+k} \uo^k \tilde P_{k,j}(\uo)
\end{equation}
for some $k$-homogeneous monogenic polynomials $\tilde P_{k,j}:\mR^n\to\mR_n$ satisfying the condition
\begin{equation}\label{Pkj_3}
\limsup_{j+k\to\infty} \sqrt[j+k]{\|\tilde P_{k,j}\|}=0.
\end{equation}
\end{itemize}
\end{corollary}

\begin{proof}
(iii)$\Rightarrow$(ii): Let $f$ have an expansion of the form \eqref{Taylor_sm}. Then we have that
$$f(x_0,p,\uo)=\sum_{j\geq k}(x_0+\uo p)^{j} \uo^k P_{k,j}(\uo)$$
where $P_{k,j}=\tilde P_{k,j-k}$. Let $a\in (0,1)$ and $C>0$ be such that $\|\tilde P_{k,j}\|\leq Ca^{j+k}$ for each $j,k\in\mN_0$. Then we have that
$\|P_{k,j}\|=\| \tilde P_{k,j-k}\|\leq Ca^{j}\leq C(\sqrt{a})^{j+k}$ for each $j,k\in\mN_0$, $j\geq k$, which finishes the proof.
Moreover, it is immediate to see that the converse (ii)$\Rightarrow$(iii) is also true.

\medskip\noindent
(i)$\Rightarrow$(iii):
Recall that $\cS\cM_0$ is introduced in Definition \ref{d-SM0}.
Let $f_0\in\underline{\cM}_0$, that is, $f_0$ is a~restriction of some entire monogenic function on $\mR^{n+1}$ to the hyperplane $x_0=0$. By Theorem \ref{t-Taylor_m}, we have that
$$f_0(\ux)=\sum_{j,k\in\mN_0} \ux^j \tilde P_{k,j}(\ux)$$
for some $k$-homogeneous monogenic polynomials $\tilde P_{k,j}:\mR^n\to\mR_n$ satisfying the condition \eqref{Pkj_3}.
Using spherical coordinates $\ux=p\uo$, we get the function
$$f_0(p\uo)=\sum_{j,k\in\mN_0} (p\uo)^{j+k} \uo^k (-1)^k \tilde P_{k,j}(\uo).$$
Moreover, there is a~unique function $f\in\cS_0$ such that $f|_{x_0=0}=f_0$, namely,
$$f(x_0,p,\uo)=\sum_{j,k\in\mN_0} (x_0+p\uo)^{j+k} \uo^k (-1)^k \tilde P_{k,j}(\uo).$$
Obviously, the converse (iii)$\Rightarrow$(i) is valid as well.
\end{proof}

Recalling that $\cS\cM_{\infty}=I_2(\cS\cM_0)$, we obtain the following result.

\begin{corollary}\label{c-SM8}
The following statements are equivalent to each other.
\begin{itemize}
\item[(i)]
A~function $f$ belongs to $\cS\cM_{\infty}$.
\item[(ii)]
A~function $f$ has a~unique expansion of the form \eqref{LaurentS} with $P_{k,j}=0$ unless $j< -k$.
\item[(iii)]
A~function $f$ has a~unique expansion of the form
\begin{equation}\label{Laurent_sm}
f(x_0,p,\uo)=\sum_{j,k\in \mN_0}(x_0+\uo p)^{-(j+k+1)} \uo^k \tilde P_{k,j}(\uo)
\end{equation}
where $\tilde P_{k,j}$ are as in \eqref{Pkj_3}.
\end{itemize}
\end{corollary}

\section{The dual Radon transform}\label{s-dRt}

Let us recall that the dual Radon transform is defined as follows.

\begin{definition}\label{defDual}
Let $\phi=\phi(p,\uo)$ be a~smooth (complex or $\mR_n$-valued) function on $ \mR\times S^{n-1}.$ For
$\ux\in\mR^n$, we define
$$
S[\phi](\ux):=\frac{1}{A_n}\int_{S^{n-1}}\phi((\ux,\uo), \uo)\;d\uo.
$$
Here $(\cdot,\cdot)$ is the scalar product on $\mR^n$, $d\uo$ is the area element on the sphere $S^{n-1}$ and $A_n$ is the area of $S^{n-1}$.
\end{definition}

It is clear that the dual Radon transform $S$ vanishes on odd functions $\phi$ (that is, when $\phi(-p,-\uo)=-\phi(p,\uo)$) and it preserves the degree of homogeneity.
Indeed, if $\phi=\phi(p,\uo)$ is a~homogeneous function of degree $\al\in\mC$ (that is, for $\lambda >0$, we have that $\phi(\la p,\uo)=\la^{\al}\phi(p,\uo)$), so is $S[\phi]$.
To prove our first main theorem we need to show the following auxiliary results.

\begin{lemma}\label{SS0}
For each $f\in\cS_0$, we have that $S[f]\in\cM_0$.
\end{lemma}

\begin{proof}
Let $f:\mR^2\times S^{n-1}\to\mR_n$ be a~smooth function, $f=f(x_0,p,\uo)$. From Definition \ref{defDual}, we get that
$$\frac{\pp}{\pp x_0}S[f]=S\left[\frac{\pp f}{\pp x_0}\right]
\mbox{\ \ and\ \ }
\frac{\pp}{\pp x_j}S[f]=S\left[\omega_j \frac{\pp f}{\pp p}(x_0,p,\uo)\right],\ j=1,\ldots,n.$$
Note that the variable $x_0$ plays a~role of a parameter with respect to the transform $S$.
Hence we obtain that $\cD (S[f])=S[\pp_{\uo} f]$, which finishes the proof.
\end{proof}

\begin{proposition}\label{dualharm} (a) Let $k\in\mN_0$ and let $P_k$ be a~homogeneous harmonic polynomial in $\mR^n$ of degree $k$.
Let $\al\in\mC$ and $\Re\al>-k-1$, where $\Re\al$ denotes the real part of $\alpha$.
Then we have that
\begin{equation}\label{eq_dualharm}
S[|p|^{\al} P_k(p\, \uo)]=B(\al,k)\;|\ux|^{\al}P_k(\ux)
\end{equation}
where the constant $B(\al,k)\in\mC$ is given by
\begin{eqnarray*}
B(\al,k) &=& \frac{2^{-(\al+k)}\;\Gamma(\al+k+1)\;\Gamma(n/2)}{\Gamma((\al+2)/2)\;\Gamma((\al+2k+n)/2)}.
\end{eqnarray*}
In particular, we have that $B(\al,k)=B(\al+2,k-1)(\al+2)/(\al+k+1)$ for $k\in\mN$
and, for $\Re\al>-k-1$, that $B(\al,k)=0$ if and only if $\al\in\{-2,-4,-6,\ldots\}$.\\ \smallskip\noindent
(b) Let $j,k\in\mN_0$, $s\in\{0,1\}$ and let $P_k$ be a $k$-homogeneous spherical monogenic in $\mR^n$. Then we have that
$S[(x_0+\uo p)^{j} \uo^k P_k(\uo)]=0$ for $j<k$ and
\begin{equation}\label{eq_dualmonog}
S[(x_0+\uo p)^{2j+s+k} \uo^k P_k(\uo)]=(-1)^k B(2j,k+s)\;X^{(2j+s)}_k(x) P_k(\ux).
\end{equation}
\end{proposition}

\begin{proof} ($a$) These formulas can be obtained using the Funk-Hecke theorem, see e.g.\ \cite[pp. 340-347]{dss}.
For the sake of completeness, we derive them from elementary properties of the dual Radon transform.
Due to orthogonal invariance of the dual Radon transform it is sufficient to show the formula of Proposition \ref{dualharm} only for the polynomial $P_k(\ux)=(x_1+ix_2)^k.$
Indeed, for each $A\in SO(n)$, we have that
\begin{equation}\label{dualinvariance}
S[\phi(p,A^{-1}\uo)](\ux)=S[\phi(p,\uo)](A^{-1}\ux).
\end{equation}
Denote $z=x_1+ix_2$ and $\nu=\omega_1+i\omega_2$. Thus we need to prove by induction on $k$ the following formula
\begin{equation}\label{dualpower}
S[|p|^{\al} p^k \nu^k]=B(\al,k)\;|\ux|^{\al} z^k.
\end{equation}
(i) For $k=0$, the formula can be obtained using the spherical coordinates (see \cite[p. 22]{VK}).
Indeed, we have that
$$\frac{1}{A_n}\int_{S^{n-1}}|\omega_1|^{\al}\;d\uo=\frac{A_{n-1}}{A_n}\int_0^{\pi}|\cos(\theta)|^{\al} \sin^{n-2}\theta \;d\theta=B(\al,0).$$

\noindent
(ii) Assume that $k\in\mN$ and the formula \eqref{dualpower} is valid for $k-1$.
Denote $\partial_{\bar z}=\frac 12 (\partial_{x_1}+i\partial_{x_2})$.
As we know, for a~smooth function $\phi$ and $j=1,\ldots,n,$ we have that
\begin{eqnarray}\label{dualder}
S\left[\omega_j \frac{\pp\phi}{\pp p}(p,\uo)\right]=\frac{\pp}{\pp x_j}S[\phi]\text{\ \ and, in particular,\ \ }
S\left[\nu \frac{\pp\phi}{\pp p}(p,\uo)\right]=2\frac{\pp}{\pp \bar z} S[\phi].
\end{eqnarray}
By this property and the induction assumption, we get that $(\al+k+1) S[|p|^{\al} p^k \nu^k]$ is equal to
$$
S\left[\nu^k \frac{\pp}{\pp p}(|p|^{\al}p^{k+1})\right]=
2\frac{\pp}{\pp\bar z} S[|p|^{\al+2}p^{k-1}\nu^{k-1}]=B(\al+2,k-1)(\al+2) |\ux|^{\al} z^k,
$$
which finishes the proof of the statement ($a$).

\medskip\noindent
($b$) It follows from ($a$) that \eqref{eq_dualmonog} is valid for $x_0=0$. Since the dual Radon transform $S$ maps slice monogenic functions into monogenic ones (see Lemma \ref{SS0}) both sides of \eqref{eq_dualmonog} are entire monogenic functions with the same values on the hyperplane  $x_0=0$ and so they coincide on $\mR^{n+1}$.
\end{proof}

Now we show our first main theorem.

\begin{thmA}
The dual Radon transform $S$ is a~linear map of the $\mathbb R_n$-module $\cS_0$ onto $\cM_0$.
Moreover  $\cS_0$ can be decomposed as
$$\cS_0 ={\rm ker}(S)\oplus \cS\cM_0.$$
In particular, $S$ is an isomorphism between $\cS\cM_0$ and $\cM_0$.
\end{thmA}

\begin{proof} By Theorem \ref{t-Laurent}, each function $f\in\cS_0$ has a~unique expansion of the form
\begin{equation}\label{Taylor_S0}
f(x_0,p,\uo)=\sum_{j,k\in \mN_0}(x_0+\uo p)^j \uo^k P_{k,j}(\uo)
\end{equation}
for some $k$-homogeneous monogenic polynomials $P_{k,j}:\mR^n\to\mR_n$ satisfying the following condition:
For each $m\in\mN_0$ and each $a>0$, there is $C>0$ such that
\begin{equation}\label{Pkj_S0}
\|P_{k,j}\|\leq C\, a^{j} (k+1)^{-m}
\end{equation}
for all $j,k\in\mN_0$.
By Corollary \ref{c-SM0}, we know that a~function $f$ belongs to $\cS\cM_0$ if and only if it
has a~unique expansion of the form \eqref{Taylor_S0} with $P_{k,j}=0$ unless $j\geq k$.
Denote by $\cK$ the submodule of functions $f\in\cS_0$ which have expansions of the form \eqref{Taylor_S0} with $P_{k,j}=0$ unless $0\leq j<k$.
Obviously, we have that $$\cS_0 =\cK\oplus \cS\cM_0.$$

We show that $\cK={\rm ker}(S)$ and $\cM_0=S(\cS\cM_0)$.
To do this, let $f\in\cS_0$ have an expansion of the form
\eqref{Taylor_S0}.
Since this expansion converges absolutely and locally uniformly on $\mR^2\times S^{n-1}$  we get, by Proposition \ref{dualharm} (b), that
\[
\begin{split}
S[f]&=\sum_{j,k\in \mN_0} S[(x_0+\uo p)^j \uo^k P_{k,j}(\uo)]=\sum_{j,k\in \mN_0} S[(x_0+\uo p)^{j+k} \uo^k P_{k,j+k}(\uo)]\\
&=\sum_{j,k\in \mN_0} c_{k,j} X^{(j)}_k(x) P_{k,j+k}(\ux)=\sum_{j,k\in \mN_0} X^{( j)}_k(x) Q_{k, j}(\ux)
\end{split}
\]
where $c_{k,j}$ are non-zero constants and $Q_{k, j}=c_{k,j}P_{k,j+k}$.
In particular, we have that $\cK={\rm ker}(S)$.
Moreover, by Lemma \ref{l-estimate}, there are strictly positive constants $C_1,C_2,a_1,a_2$ such that
$$C_1 a_1^{j+k}\leq |c_{k,j}|\leq C_2 a_2^{j+k}$$ for all $j,k\in\mN_0$.
Hence, by \eqref{Pkj_S0}, for each $a>0$ (and $m=0$)
there is $C>0$ such that
$$\|Q_{k, j}\|=|c_{k,j}|\|P_{k,j+k}\|\leq C\, (aa_2)^{j+k}$$
for all $j,k\in\mN_0$.
By Theorem \ref{t-Taylor_m} and Corollary \ref{c-SM0}, we have that $\cM_0=S(\cS\cM_0)$.
\end{proof}

In the proof of Theorem A, we characterized ${\rm ker}(S)$ as follows.

\begin{corollary}\label{c-kerS}
A~function $f$ belongs to the kernel ${\rm ker}(S)$ if and only if it has a~unique expansion of the form
$$f(x_0,p,\uo)=\sum_{k>j\geq0}(x_0+\uo p)^j \uo^k P_{k,j}(\uo)$$
for some $k$-homogeneous monogenic polynomials $P_{k,j}:\mR^n\to\mR_n$ satisfying the condition \eqref{Pkj_S0}.

Furthermore, a~function $f$ belongs to $I_2({\rm ker}(S))$ if and only if it has a~unique expansion of the form
$$f(x_0,p,\uo)=\sum_{k>j\geq0}(x_0+\uo p)^{-(j+1)} \uo^k P_{k,j}(\uo)$$
where $P_{k,j}$ are as in \eqref{Pkj_S0}.
\end{corollary}

\section{The Radon transform}\label{s-Rt}

We shall use the version of the Radon transform introduced in \cite[p. 2]{GGV}.

\begin{definition}\label{defRadon}
Let $f$ be a (complex or $\mR_n$-valued) function defined in $\mR^n.$
For a non-zero vector $\uxi\in \mR^n$ and for $p\in\mR,$ we define
$$
R[f](p,\uxi):=\int_{L(\uxi,p)} f \;d\sigma
$$
whenever the integral exists.
Here $L(\uxi,p)=\{\ux\in\mR^n|\ (\ux,\uxi)=p\}$ and the $(n-1)$-form $d\sigma $ is uniquely determined on the hyperplane $L(\uxi,p)$ by the property $d(\ux,\uxi)\wedge d\sigma =d\ux$.
\end{definition}

For basic properties of the Radon transform, see \cite[Sect. 1.3]{GGV}. In particular,
 the Radon image $R[f]$ is an even function of homogeneity $-1$, that is, $R[f](p,\uxi)=R[f](-p,-\uxi)$ and, for $\lambda\in\mR\setminus\{0\}$,
we have that $R[f](\lambda p,\lambda\uxi)=|\lambda|^{-1}R[f](p,\uxi)$.
So the Radon image $R[f]$ is uniquely determined by its values on $\mR_{+}\times S^{n-1}$.
For homogeneous functions, we can say even more.

\begin{lemma}\label{radonhomog}
Let $\al\in\mC$ and let $f$ be a~function defined in $\mR^n$ and homogeneous of degree $\alpha$, that is, for $\lambda >0$, we have that $f(\lambda\ux)=\lambda^{\al}f(\ux)$.
Then, for $p>0$ and $\uxi\not=0$, it holds that
$$R[f](p,\uxi)=\frac{p^{n-1+\al}}{|\uxi|^{n+\al}} R[f](1,\uo),$$
provided that $\uo=\uxi/|\uxi|$ and the value $R[f](1,\uo)$ of the Radon transform makes sense.
\end{lemma}

\begin{proof}
For the Radon image $R[f]$ is a~function of homogeneity $-1$ in the variables $(p,\uxi)$, it is sufficient to prove the result for $p>0$ and $|\uxi|=1$.
In this case, we have that
$$R[f](p,\uxi)=\int_{L(\uxi,0)} f(p\uxi+\uy) \;d\sigma(\uy)=p^{n-1}\int_{L(\uxi,0)} f(p\uxi+p\uy) \;d\sigma(\uy)=
p^{n-1+\al}R[f](1,\uxi),$$
which finishes the proof.
\end{proof}

To show our second main theorem we need the following results.

\begin{lemma}\label{RM8}
For each $f\in\cM_{\infty}$, we have that $R[f]\in\cS_{\infty}$.
\end{lemma}

\begin{proof}
Let $f\in\cM_{\infty}$. By \cite[(c), p. 6]{GGV}, we have that
$$\frac{\pp}{\pp x_0}R[f]=R\left[\frac{\pp f}{\pp x_0}\right]
\mbox{\ \ and\ \ }
\omega_j \frac{\pp }{\pp p} R[f]=R\left[\frac{\pp f}{\pp x_j}\right ],\ j=1,\ldots,n.$$
Note that the variable $x_0$ plays a~role of parameter with respect to the transform $R$.
Hence we obtain that $\pp_{\uo} (R[f])=R[\cD f]=0$, which finishes the proof.
\end{proof}

\begin{proposition}\label{radonharm} (a) Let $k\in\mN_0$ and let $P_k$ be a~homogeneous harmonic polynomial in $\mR^n$ of degree $k$.
Let $\al\in\mC$ and $\Re\al>k$.
Then, for $p\not=0$ and $\uxi\not=0$, we have that
\begin{equation}\label{eq_radonharm}
R\left[\frac{P_k(\ux)}{|\ux|^{\al +n-1}}\right]=A(\al,k)\;\frac{P_k(p\,\uxi)\;|\uxi|^{\al-2k-1}}{|p|^{\al}}
\end{equation}
where the constant $A(\al,k)\in\mC$ is given by
\begin{equation*}
A(\al,k) = \frac{2^{k-\al+1}\pi^{n/2}\;\Gamma(\al-k)}{\Gamma((\al+n-1)/2)\;\Gamma((\al-2k+1)/2)}.
\end{equation*}
In particular, we have that $A(\al,k)=A(\al,k-1)(\al-2k+1)/(\al-k)$ for $k\in\mN$
and, for $\Re\al>k$, that $A(\al,k)=0$ if and only if $\al\in\{1,3,5,\ldots,2k-1\}$. \\ \smallskip\noindent
(b) Let $j,k\in\mN_0$, $s\in\{0,1\}$, $2j+s\geq 1$ and let $P_k$ be a~homogeneous monogenic polynomial in $\mR^n$ of degree $k$.
Then we have that
\begin{equation}\label{eq_radonmonog}
R\left[\frac{\bar x\; X_k^{(2j+s-1)}(x)P_k(\ux)}{|x|^{n+1+2(2j+s-1+k)}}\right]=(-1)^{s+1} A(2j+2k+2s,k+s)\;(x_0+\uo p)^{-(2j+s+k)}\uo^k P_k(\uo).
\end{equation}
\end{proposition}

\begin{proof}
($a$) Due to orthogonal invariance of the Radon transform it is sufficient to show the formula \eqref{eq_radonharm} only for the polynomial $P_k(\ux)=(x_1+ix_2)^k.$
Indeed, let $A$ be a~rotation of the Euclidean space $\mR^n$, that is, $A\in SO(n)$. Then, as it is well-known, we have that
\begin{equation}\label{radoninvariance}
R[f(A^{-1}\ux)](p,\uxi)=R[f](p,A^{-1}\uxi).
\end{equation}
Moreover, let $\cH_k(\mR^n)$ denote the set of harmonic polynomials $P_k$ in $\mR^n$ which are homogeneous of degree $k$.
Then the set $\cH_k(\mR^n)$ forms an irreducible module under the action of the group $SO(n)$ and $(x_1+ix_2)^k\in\cH_k(\mR^n)$.

Furthermore, we restrict ourselves to the case when $p>0.$   Denote $z=x_1+ix_2$ and $\nu=\xi_1+i\xi_2$. Thus we need to prove by induction on $k$ the following formula
\begin{equation}\label{radonpower}
R\left[\frac{z^k}{|\ux|^{\al+n-1}}\right]=A(\al,k)\;\frac{{\nu}^k\;|\uxi|^{\al-2k-1}}{p^{\al-k}}.
\end{equation}
(i) For $k=0$, the formula is known (see \cite{GGV}). It also follows from Lemma \ref{radonhomog}. Indeed, for $\uo\in S^{n-1}$, we have that
$$R[|\ux|^{-(\al+n-1)}](1,\uo)=A_{n-1}\int_0^{\infty} t^{n-2} (1+t^2)^{-(\al+n-1)/2}\; dt=A(\al,0).$$

\noindent
(ii) Assume that $k\in\mN$ and the formula \eqref{radonpower} is valid for $k-1$.
Denote $\partial_{\bar\nu}=(1/2)(\partial_{\xi_1}+i\partial_{\xi_2})$.
For $j=1,\ldots,n,$ we have that (see \cite[(d), p. 7]{GGV})
\begin{eqnarray}\label{der}
\frac{\pp}{\pp p}(R[x_j\,f(\ux)])=-\frac{\pp}{\pp\xi_j}(R[f]).
\end{eqnarray}
By this property and the induction assumption, we get that
$$\frac{\pp}{\pp p}\; R\left[z\; \frac{z^{k-1}}{|\ux|^{\al+n-1}}\right]=-2\partial_{\bar\nu}\; R\left[\frac{z^{k-1}}{|\ux|^{\al+n-1}}\right]=
-2\partial_{\bar\nu}\;\left[A(\al,k-1)\;\frac{\nu^{k-1}\;|\uxi|^{\al-2k+1}}{p^{\al-k+1}}\right]=$$
$$=-A(\al,k-1)(\al-2k+1)\;\frac{\nu^{k}\;|\uxi|^{\al-2k-1}}{p^{\al-k+1}}.$$
Moreover, by integrating with respect to the variable $p$, we obtain the following expression
$$R\left[\frac{z^k}{|\ux|^{\al+n-1}}\right]=A(\al,k-1)\frac{\al-2k+1}{\al-k}\;\frac{\nu^{k}\;|\uxi|^{\al-2k-1}}{p^{\al-k}}+C(\uxi).$$
Finally, Lemma \ref{radonhomog} shows that $C(\uxi)=0$,
which finishes the proof of the statement ($a$).

\medskip\noindent
($b$) It follows easily from ($a$) that \eqref{eq_radonmonog} is valid for $x_0=0$. Since the Radon transform $R$ maps monogenic functions into slice monogenic ones (see Lemma \ref{RM8}) both sides of \eqref{eq_radonmonog} are slice monogenic functions of $\cS_{\infty}$ with the same values on the hyperplane  $x_0=0$ and so they coincide on $(\mR^2\setminus\{0\})\times S^{n-1}$.
\end{proof}

Now we are ready to prove our second main theorem.

\begin{thmB}
The Radon transform $R$ is an injective linear map of the $\mathbb R_n$-module $\cM_\infty$ into  $\cS_\infty$.
Denote $\cS\cM_{\infty}=I_2(\cS\cM_0)$. Then $R(\cM_\infty)=\cS\cM_{\infty}$
and $R$ is an isomorphism between $\cM_\infty$ and $\cS\cM_{\infty}$.
Moreover, we have that
$$\cS_\infty=I_2({\rm ker}(S))\oplus \cS\cM_{\infty}.$$
\end{thmB}

\begin{proof}
As we know (see \eqref{e-Laurent_m}) a~given function $f\in{\cM}_\infty$  has a~unique expansion of the form
\begin{equation}\label{e-Laurent_m_2}
f(x)=\frac{\bar x}{|x|^{n+1}}\sum_{j,k\in \mN_0} \frac{X^{(j)}_k(x)  P_{k,j}(\underline{x})}{|x|^{2(j+k)}}
\end{equation}
for some $k$-homogeneous monogenic polynomials $P_{k,j}:\mR^n\to\mR_n$ satisfying
\begin{equation}\label{Pkj_4}
\limsup_{j+k\to\infty} \sqrt[j+k]{\|P_{k,j}\|}=0.
\end{equation}
Then we have
\begin{equation}\label{e-Rsum}
R[f]=\sum_{j,k\in \mN_0} R[f_{k,j}]
\mbox{\ \ with\ }f_{k,j}(x)=\frac{\bar x\;   X^{(j)}_k(x)P_{k,j}(\underline{x})}{|x|^{n+1+2(j+k)}}.
\end{equation}
Indeed, we can use Lebesgue's dominated convergence theorem because
$$\sum_{j,k\in \mN_0} R[|f_{k,j}|]<\infty.$$
Namely, for some $b>0$ and for each $j,k\in\mN_0$, we have that $|X^{(j)}_k(x)|\leq b^{k+j} |x|^j$, see Lemma \ref{l_est_EF}.
Finally, for each $j,k\in\mN_0$, we easily obtain
$$|f_{k,j}(x)|\leq \frac{b^{k+j}\|P_{k,j}\|_{\infty}}{|\ux|^{n}}\mbox{\ \ and\ \ }
R[|\ux|^{-n}](x_0,p,\uo)<\infty,\ \ \  (x_0,p,\uo)\in(\mR^2\setminus\{0\})\times S^{n-1} .$$
Hence, by \eqref{e-Rsum} and Proposition \ref{radonharm} (b), we get that
$$R[f]=\sum_{j,k\in \mN_0} d_{k,j}\;(x_0+\uo p)^{-(j+1+k)}\uo^k P_{k,j}(\uo).$$
In particular, it is immediate to see that $R[f]=0$ only if $f=0$.
Moreover, by Lemma \ref{l-estimate}, there are strictly positive constants $C_1,C_2,a_1,a_2$ such that
$$C_1 a_1^{j+k}\leq |d_{k,j}|\leq C_2 a_2^{j+k}$$ for all $j,k\in\mN_0$.
By Corollary \ref{c-SM8}, we have that $R[f]\in\cS\cM_{\infty}$ and, obviously,we have that $R(\cM_{\infty})=\cS\cM_{\infty}$.
\end{proof}


\end{document}